\documentclass[12pt]{amsart}
\usepackage{amsmath}
\usepackage[small,bf]{caption}
\usepackage{graphicx}
\usepackage{amsmath}
\usepackage{verbatim}
\usepackage[dvipsnames, usenames]{color}
\usepackage[colorlinks=true, pdfstartview=FitV, linkcolor=BrickRed, citecolor=BrickRed, urlcolor=BrickRed]{hyperref}

\begin{document}

\numberwithin{equation}{section}

\newtheorem{thm}[subsection]{Theorem}
\newtheorem{lem}[subsection]{Lemma}
\newtheorem{cor}[subsection]{Corollary}
\newtheorem{prop}[subsection]{Proposition}
\newtheorem{obs}[subsection]{Observation}
\newtheorem{claim}[subsection]{Claim}
\newtheorem{con}[subsection]{Conjecture}
\newtheorem{Q}[subsection]{Question}

\theoremstyle{definition}
\newtheorem{definition}[subsection]{Definition}
\newtheorem{rem}[subsection]{Remark}
\newtheorem{eg}[subsection]{Example}
\newtheorem{egs}[subsection]{Examples}

\newcommand{\R}{\mathbf{R}}
\newcommand{\C}{\mathbf{C}}
\newcommand{\US}{\mathbf{S}}
\newcommand{\Z}{\mathbf{Z}}
\newcommand{\N}{\mathbf{N}}

\def\a{\alpha}
\def\b{\beta}
\def\g{\gamma}
\def\G{\Gamma}
\def\lam{\lambda}
\def\th{\theta}
\def\Th{\Theta}
\def\eps{\varepsilon}
\def\d{\delta}
\def\i{\mathrm{i}}
\def\k{\kappa}
\def\w{\omega}
\def\z{\zeta}

\def\half{\frac{1}{2}}
\def\nv#1{\mathbf{#1}}
\def\bd{\partial}
\def\co{\colon}
\def\of{\circ}
\def\dist{\mathrm{d}}
\def\l{\left}
\def\r{\right}

\def\fib{\hfill\break}
\def\ds{\displaystyle}
\def\bitem#1{\item\textbf{#1}}
\def\eob{\hfill\qedsymbol}

\def\and{\quad\text{and}\quad}
\def\givenby{\quad\text{given by}\quad}
\def\Pf#1{\noindent{\footnotesize{\bf Proof.} #1\hfill//$\,$}}

\def\fe{\texttt{8}}
\def\cyl{\mathrm{Cyl}}
\def\cx{\textsc{cx}}
\def\stack{\mathcal{T}}
\def\c{\mu}
\def\slab{\mathcal{S}}
\def\ph{\phantom}
\def\:{\colon}

\def\tit{\textit}
\def\tbt{\textbf}

\long\def\symbolfootnote[#1]#2{
	\begingroup
	\def\thefootnote{\fnsymbol{footnote}}\footnote[#1]{#2}
	\endgroup}
	
%%%
%%%  For marginal notes. Remove this definition before submitting
%%%

\catcode`@=11 \@mparswitchfalse  %This puts the \mnote's on the right.

\newcounter{mnotecount}[page]
\renewcommand{\themnotecount}{\arabic{mnotecount}}

\newcommand{\mnote}[1]
{\protect{\stepcounter{mnotecount}}$^{\mbox{\footnotesize  $%\!\!\!\!\!\!\,
      \bullet$\themnotecount}}$
\marginpar{\null\hskip-9pt\raggedright\tiny\em
\themnotecount:\! \it #1} }

%%%
%%%  End of marginal note defn.
%%%

\parskip 5pt
\parindent 0pt
\baselineskip 16pt

%%%%%%%%%%%%%%%%% Content begins here %%%%%%%%%%%%%%%%

\author{Bruce Solomon}
%\address{Math Department, Indiana University, Bloomington IN 47405}
%\email{solomon@indiana.edu}
%\urladdr{mypage.iu.edu/$\sim$solomon}

%\subjclass{53A04}
%\keywords{}
%\date{Begun: February 20, 2014. ArXiv'd September 16, 2015. Last Typeset \today.}
%\thanks{Supported by \dots.}

\begin{abstract}
	We prove: \tit{If a complete connected $\,C^{2}\,$ surface $\,M\,$ in $\,\R^{3}$ has 
	general position, intersects some plane along a clean figure-8 (a loop with total 
	curvature zero) and all compact intersections with planes have central symmetry, then
	$\,M\,$ is a (geometric) cylinder over some central figure-8}. On the way, we 
	establish interesting facts about centrally symmetric loops in the plane; 
	for instance, a clean loop with even rotation number $\,2k\,$ can never be central 
	unless it passes through its center exactly twice and $\,k=0\,$.
\end{abstract}

\title[\dots\ figure-8 cross-cuts\ \dots]{Central figure-8 cross-cuts make surfaces cylindrical}
\maketitle

\thispagestyle{empty}

%\hrule
\vskip 10pt

\begin{center}
	\hyperref[sec:intro]{{\S1}}$\,\ \diamond\ $
	\hyperref[sec:RS]{{\S2}}$\,\ \diamond\,\ $ 
	\hyperref[sec:MR]{{\S3}}$\,\ \diamond\,\ $ 
	\hyperref[sec:bib]{{Ref.}}
\end{center}
\vskip8pt
%\hrule

\section{Introduction}\label{sec:intro}
A set $\,X\subset\R^{n+1}\,$ has a \tit{center} $\,c\in\R^{n+1}$ (or has \tit{central symmetry},
or \tit{is central}) if the $c$-fixing reflection $\,x\mapsto 2 c-x\,$ maps $\,X\,$ 
to itself.

\tit{What can one say about a set $\,X\subset\R^{n+1}$ that meets every hyperplane 
along a central set?}

When $\,P\,$ is a hyperplane, we (for now) call $\,X\cap P\,$ a \tit{cross cut} of $\,X\,$.
Later we define cross-cut more narrowly.

\tit{Are cross-cuts of central sets always central?} Not generally, unless they go
through the center. A cube in $\,\R^{3}$ is central, for instance, but a plane that severs 
its corner cuts it along a triangle, which is never central.

\tit{Do central cross-cuts make a set central?} Not in $\,\R^{2}$. For instance, 
all cross-cuts of a plane triangle are (trivially) central, but again, 
no triangle is central. 

When $\,n+1>2$, however, we know of no such counterexample. Indeed, in the presence of, say, 
convexity, central cross-cuts can force \tit{more} than just centrality. 

For example,  when $\,K\subset\R^{n+1}$ is a convex body and $\,n+1>2$, 
central cross-cuts force $\,K\,$ to be ellipsoidal. 
The most general formulation of this fact was proven by S.~P.~Olovjanischnikoff 
\cite{olov},\footnote{G.~R.~Burton gives a nice statement of Olovjanischnikoff's result in 
\cite[Lemma 3]{bu}.} who relaxed restrictions (e.g., on smoothness) in earlier results of 
this type by Brunn and Blaschke (see \cite[\S44, \S84]{VuD} and \cite{wb}).

In \cite{sor}, we drew a similar conclusion for (not necessarily convex) hypersurfaces
of revolution in $\,\R^{n+1}$. If their compact convex cross-cuts are central, they must be 
quadric: ellipsoids, hyperboloids, paraboloids, or circular cylinders. 

We later used that fact in \cite{cpo} to get a broader result: 
when a complete immersed surface in $\,\R^{3}\,$ has a connected compact transverse cross-cut, 
and all cross-cuts of that type are central, uniformly convex \tit{ovals}, the surface is 
either a \hyperlink{tgt:ccyl}{central cylinder} or a \hyperlink{tgt:quad}{tubular quadric} 

None of results, however, manages to exploit centrality of cross-cuts
without also requiring their convexity. Here for the first time, we drop the convexity 
requirement, replacing it with an admittedly special but very different alternative. 
We consider surfaces in $\,\R^{3}\,$ whose cross-cuts are \tit{clean} (meaning they
never visit a point twice tangent to the same line) and have total geodesic curvature 
zero, making them \tit{figure-8's} up to regular homotopy. Our main result, 
\hyperref[thm:8case]{Theorem \ref{thm:8case}} says that a surface with this property 
must be a central cylinder. \hyperref[sec:MR] {Section \ref{sec:MR}} of our paper focuses
on the proof of that fact.

\hyperref[sec:RS]{Section \ref{sec:RS}} (which we find interesting on its own) is devoted 
mainly the proof of a simple but critical ingredient: \tit{Any clean, 
central figure-8 must visit its center exactly twice}. The key role this plays in 
\hyperref[sec:MR]{Section \ref{sec:MR}} is explained in the paragraphs
immediately below the statement of \hyperref[thm:8case]{Theorem \ref{thm:8case}}.
In proving the supporting fact, however, we get the general theory summarized in 
\hyperref[prop:oo0]{Proposition \ref{prop:oo0}}, which says, in part, that a clean 
central loop must either have \tit{odd} rotation index, and avoid its center entirely, 
or else be a figure-8  (index {zero}) that visits its center exactly twice. 

Simple examples---the unit circle traced twice, for instance, or the loop in 
\hyperref[fig:unclean]{Figure \ref{fig:unclean}}---show that such statements fail for loops 
that are not cleanly immersed. The reasoning in both \hyperref[sec:RS]{\S\ref{sec:RS}} and 
\hyperref[sec:MR]{\S\ref{sec:MR}} would simplify considerably if we didn't need to assume 
and exploit general position arguments to exclude unlikely ``pathologies'' like these.

\section{Reparametrization and symmetry}\label{sec:RS}

\begin{definition}[Central symmetry]\label{def:central}
	 An \tit{immersion} $\,F\colon M\to\R^{n+1}\,$ is central when its image has central symmetry.
\end{definition}

\begin{definition}[Reparametrization]
	When $\,\a,\b\:\US^{1}\to\R^{2}$ are immersed loops, we say that $\,\b\,$
	\tit{repara\-me\-trizes} $\,\a\,$ when $\,\b = \a\of\phi\,$ for some diffeomorphism 
	$\,\phi\co\US^{1}\to\US^{1}$. It \tit{preserves} or \tit{reverses} orientation 
	when $\,\phi\,$ preserves or reverses the orientation of the circle.\eob
\end{definition}

The following fact will pester us:
\tit{Two immersed loops with the same image don't always repara\-metrize each other, 
even if they visit each point equally often.} Centrality doesn't mitigate this inconvenient
truth, as discussed with regard to \hyperref[fig:unclean]{Figure 
\ref{fig:unclean}} below.

\begin{egs}\label{eg:unclean}
	The unit circle $\,\US^{1}\subset\C\,$ is central about the origin. If we parametrize 
	it as usual by $\,t\mapsto e^{\i\,t}$, reflection through the origin produces the 
	orientation-\tit{preserving} repara\-me\-trization $\,t\mapsto -e^{\i t}$.
	
	Contrastingly, consider the figure-8 parametrized by $\,t\mapsto(\cos t,\sin2t)$. 
	While likewise central about the origin, reflection through the origin induces 
	the orientation-\tit{reversing} reparametrization $\,t\mapsto(\cos t,-\sin 2t)$.
	(\hyperref[fig:egO8]{Figure \ref{fig:egO8}}).
	
	\begin{figure}[h]
		\begin{center}
			\includegraphics[height=40mm]{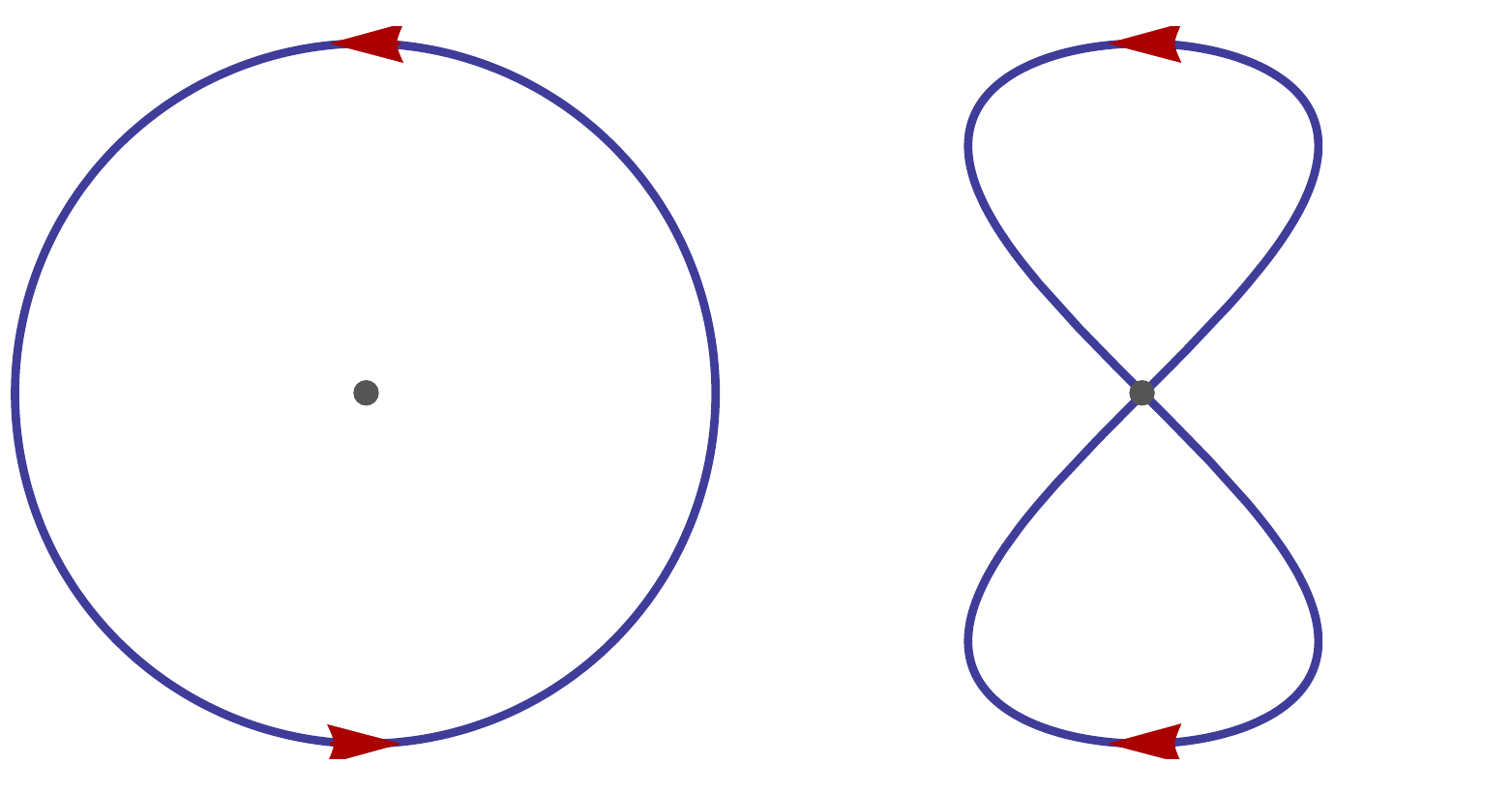}
			\caption{Reflection through the origin reparametrizes both the unit circle and 
			figure-8. While pre\-serving orientation on the circle, however, it \tit{reverses} 
			orientation on the figure-8.}
			\label{fig:egO8}
		\end{center}
	\end{figure}	
	
	In \hyperref[fig:unclean]{Figure \ref{fig:unclean}} however, we depict a smooth, 
	origin-central immersion $\,\a\colon \US^{1}\to\R^{2}$ whose 
	reflection $\,-\a\,$ does \tit{not} reparametrize $\,\a$, even though $\,\a\,$ and 
	$\,-\a\,$ have the same image. To see this, orient the open sets $\,U_{L},\, U_{0}$ 
	and $U_{R}\,$ there in the standard way, and observe that\smallskip
	\begin{align*}
		\a&=\bd(U_{0}+U_{R}-U_{L})\\
		-\a&=\bd(U_{0}-U_{R}+U_{L})
	\end{align*}
	As the oriented domains bounded by $\,\a\,$ and $\,-\a\,$ are neither equal nor 
	opposite, $\,-\a\,$ neither preserves, nor reverses the orientation of $\,\a$. 
	It evidently does not, therefore, reparametrize.
\eob
\end{egs}

	\begin{figure}[h]
		\begin{center}
			\includegraphics[height=25mm]{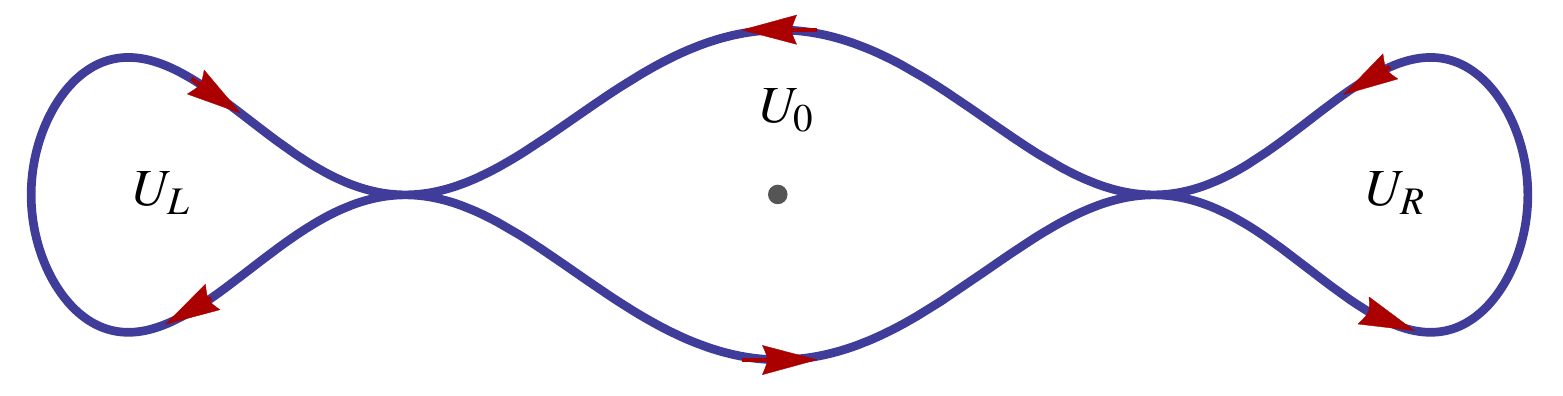}
			\caption{This immersed loop $\,\a\,$
			winds counter-clockwise around $\,U_{R}\,$ and $\,U_{0}$, but clockwise around 
			$\,U_{L}\,$. Reflection through the center preserves its image, but 
			$\,-\a\,$  is not a reparametrization of $\,\a$.}
			\label{fig:unclean}
		\end{center}
	\end{figure}

We can exclude behavior like that depicted in \hyperref[fig:unclean]{Figure 
\ref{fig:unclean}} by requiring our loops to be \tit{cleanly immersed}: 

\begin{definition}[Double-points/clean loops]\label{def:cdp}
	A point $\,p\,$ in the image of an immersed curve $\,\a\,$ is a \tit{double-point} 
	when its preimage contains more than one point. When it contains {exactly} \tit{two}, 
	we call it a \tit{simple} double-point.

	An immersion $\,\a\colon \US^{1}\to\R^{2}$ has \tit{clean} double-points (or 
	\tit{is clean}) if, whenever $\,t_{1},t_{2}\in\US^{1}$ are distinct preimages of a 
	single point in $\,\R^{2}$, they have respective neighborhoods $\,U_{1}\,$ and 
	$\,U_{2}\,$ whose images $\,\a(U_{1})\,$ and $\,\a(U_{2})\,$ intersect transversally.
\eob
\end{definition}

\begin{rem}
	Though more familiar, a \tit{general position} assumption (like the one we use at the 
	start of \hyperref[sec:MR]{\S\ref{sec:MR}} below) would be more restrictive than that of
	\tit{clean double-points} for loops in $\,\R^{2}$. The latter lets a loop pass 
	through a single point three or more times as long as no two velocity vectors there 
	are collinear. General position would prohibit triple intersections.
\eob
\end{rem}

The double-points in \hyperref[eg:unclean]{Example \ref{eg:unclean}} are obviously 
\tit{not} clean, and we shall see that when two loops with the same image do \tit{not} 
reparametrize each other, they \tit{must} have unclean double-points. Indeed, the 
principal result of this section, \hyperref[prop:repo]{Proposition \ref{prop:repo}}, and 
the main facts leading up to it all fail without cleanness, as consideration of 
\hyperref[fig:unclean]{Figure \ref{fig:unclean}} quickly reveals.

We will denote the intrinsic distance between points $\,s_{1},s_{2}\,$  in the sphere
of any dimension (here the circle) by $\,\phi(s_{1},s_{2})$. We write $\,\k_{g}\,$ for
the \tit{geodesic} (or \tit{signed}) \tit{curvature} along a loop $\,\a\:\US^{1}\to\R^{2}$.
It is given by
\begin{equation}\label{eqn:kg}
	\k_{g}(t):=\frac{\det(\a',\a'')}{\l|\a'\r|^{3}}
\end{equation}
\medskip

\begin{obs}\label{obs:kbar}
	Suppose we have a $\,C^{2}\,$ unit-speed loop $\,\a\colon \US^{1}\to\R^{2}$. If 
	$\,\bar\k:=\max_{\US^{1}}|\k_{g}|$, and $\,\a^{-1}(p)\,$ contains distinct inputs 
	$\,s_{1},s_{2}\in\US^{1}$ for some $\,p\in\R^{2}$, then 
	$\,\phi(s_{1},s_{2})\ge\pi/\bar\k$.
\end{obs}

\begin{proof}
	In either arc $\,A\,$ joining $\,\a_{1}\,$ to $\,\a_{2}\,$ in $\,\US^{1}$, some point 
	$\,S\in A\,$ maximizes $\,|\a(s)-p|^{2}\,$ among $\,s\in A$, and $\,\dot\a(S)\,$ is 
	then \tit{perpendicular} to $\,\a(S)-p$. At the same time, any non-trivial linear 
	function that vanishes on $\,\a(S)-p\,$ will attain at least one local extremum on each 
	component of $\,A\setminus\{S\}$, say at points $\,\a(s_{-})\,$ and $\,\a(s_{+})\,$ 
	respectively. Since $\,\dot\a\,$ must be \tit{parallel} to $\,\a(S)-p\,$ at these 
	points, the intervals $\,(s_{-},S),\,(S,s_{+})\subset A\,$ both map to arcs with 
	total absolute curvature at least $\,\pi/2\,$. As $\,\a\,$ has unit speed, we 
	may then deduce\medskip
	\[
		\th\l(s_{1},s_{2}\r)\bar\k
		\ge 
		\th\l(s_{-},s_{+}\r)\bar\k
		\ge 
		\int_{s_{-}}^{s_{+}}|\k_{g}(s)|\ ds 
		\ge 
		\pi
	\]
\end{proof} 

In general, a $\,C^{2}\,$ immersion $\,\US^{1}\to\R^{2}$ can have infinitely many double 
points, even without retracing any open arc along its image. Not so for clean immersions:

\begin{lem}\label{lem:clean}
	A clean $\,C^{2}\,$ immersion $\,\a\colon \US^{1}\to\R^{2}$ has at most finitely many 
	double-points, and at any double-point $\,p$, there is an $\,\eps=\eps(p)>0\,$ for 
	which $\,\a^{-1}(B(p,\eps))\,$ is a finite union of embedded arcs passing through 
	$\,p\,$ with pairwise distinct tangent lines.
\end{lem}

\begin{proof}
	With no loss of generality, assume $\,\a\,$ has unit speed. Set 
	$\,\bar\k:=\,\max_{\US^{1}}|\k_{g}|\,$ as in the \hyperref[obs:kbar]{Observation}
	above.
	
	Suppose (toward a contradiction) that $\,\a\,$ had infinitely many double 
	points. Since $\,\US^{1}$ and $\,\a(\US^{1})\,$ are both compact, that would imply the 
	existence of a cluster point $\,p\in\a(\US^{1})\,$, along with convergent sequences 
	$\,(s_{n}), (s_{n}')\subset\US^{1}$ with\smallskip
	\[
		s_{n}\ne s_{n}'
		\quad\text{and}\quad
		\a(s_{n})=\a(s_{n}')\to p
	\]
	and yet $\,\a(s_{n})\ne p\,$ for all $\,n\in\N\,$. 
	
	\hyperref[obs:kbar]{Observation \ref{obs:kbar}} ensures 
	$\,|s_{n}-s_{n}'|>\pi/\bar\k$, so the respective  limits $\,s\,$ and $\,s'\,$ of 
	these sequences must obey that same estimate. In particular, $\,s\ne s'$. By continuity, 
	however, $\,\a(s)=\a(s')$, which forces the collinearity of \smallskip
	\[
		\frac{\a(s_{n})-\a(s)}{s_{n}-s}
		\quad\text{and}\quad
		\frac{\a(s_{n}')-\a(s')}{s_{n}'-s'}
	\]
	
	for each $\,n$. Letting $\,n\to\infty$, we see that $\,\dot\a(s)\,$ and $\,\dot\a(s')$
	must also be collinear. This contradicts our assumption of clean double-points.
	So $\,\a\,$ has at most finitely many double-points.
	
	To prove the remaining claim, we note that since \hyperref[obs:kbar]{Observation
	\ref{obs:kbar}} puts a lower bound on the distance between any two points in 
	$\,\a^{-1}(p)$, the compactness of $\,\US^{1}$ makes $\,\a^{-1}(p)\,$ 
	finite. By definition of \tit{immersion}, the inverse function theorem then 
	yields the asserted $\,\eps(p)>0$, while that of \tit{clean} makes tangent 
	lines pairwise distinct at $\,p$.
\end{proof}

\begin{lem}\label{lem:repo}
	Suppose $\,\a,\b\co\US^{1}\to\R^{2}$ are clean unit-speed $\,C^{2}$ loops with the 
	same image. Suppose $\,p=\b(t_0)\,$ is a double-point of $\,\a$, and $\,\eps>0\,$ 
	is small enough to make $\,\a^{-1}(B(p,\eps))\,$ a union of embedded arcs with 
	distinct tangent lines at $\,p$, as provided by \hyperref[lem:clean]{Lemma 
	\ref{lem:clean}}. Then one such arc contains $\,\b(t_{0}-\d, t_{0}+\d)\,$ for all 
	sufficiently small $\,\d>0$.
\end{lem}

\begin{proof}
	Take $\,\eps>0\,$ small enough to satisfy the hypothesis of \hyperref[lem:clean]
	{Lemma \ref{lem:clean}}, and let $\,A_{1}, A_{2}, \dots, A_{k}\,$ denote the 
	(distinct) arcs whose union then constitutes $\,\a^{-1}(B(p,\eps))\,$. Define 
	$\,I_{n}:=(t_{0}-\frac{1}{n},\,t_{0}+\frac{1}{n})\,$ for $\,n\in\N$. When $\,n\,$ is
	large, $\,\b\,$ embeds $\,I_{n}$, and since $\,\b\,$ and $\,\a\,$ have the same image, 
	$\,\b(I_{n})\,$ must then lie in the \tit{union} of the $\,A_{i}$'s.
	
	If for every such $\,n$, we could find $\,t_{n}, t_{n}'\ne t\,$ in $\,I_{n}\,$
	with $\,\b(t_{n})\,$ and $\,\b(t_{n}')\,$  in \tit{different} $\,A_{i}$'s, we could
	renumber the $\,A_{i}$'s and pass to a subsequence to arrange $\,\b(t_{n})\in A_{1}\,$ 
	and $\,\b(t_{n}')\in A_{2}\,$ for all large $\,n$. 
	But $\,\lim_{n\to\infty} t_{n}= \lim_{n\to\infty} t_{n}'=t_{0}$, and $\,\b\,$ is 
	differentiable, so computing $\,\dot\b(t_0)\,$ on the two different sequences would 
	give the same result, forcing the tangent lines to $\,A_{1}\,$ and $\,A_{2}\,$ at 
	$\,p=\b(t_0)\,$ to agree. This would contradict the last assertion of 
	\hyperref[lem:clean]{Lemma \ref{lem:clean}}. So $\,\b(I_{n})\,$ must stay in one 
	$\,A_{i}$, as claimed.
\end{proof}

\begin{definition}
	By the \tit{lift} of an immersed unit-speed arc $\,\a\co(a,b)\to\R^{2}$, we mean the 
	arc parametrized by $\,s\mapsto\left(\a(s),\dot\a(s)\right)\,$ in the unit tangent
	bundle $\,\R^{2}\times\US^{1}$.
	\eob
\end{definition}

Using \hyperref[lem:clean]{Lemma \ref{lem:clean}}, the reader will easily verify

\begin{obs}\label{obs:lift}
	If $\,\a\co\US^{1}\to\R^{2}$ is a cleanly immersed loop, its lift is embedded.
	The lift of any reparametrization $\,\a\of\phi\,$ either reparame\-trizes that of 
	$\,\a$, or never meets it, depending on whether $\,\phi\,$ preserves or 
	reverses orientation respectively.
\end{obs}

We can now prove the fact that makes the main results of this section accessible.

\begin{prop}\label{prop:repo}
	Suppose $\,\a,\b\co\US^{1}\to\R^{2}$ are clean, unit-speed $\,C^{2}\,$ immersions with 
	the same image. Then $\,\b\,$ reparametrizes $\,\a$, and the two loops have the 
	same orientation if and only if their lifts meet.
\end{prop}

\begin{proof} 
	By \hyperref[obs:lift]{Observation \ref{obs:lift}}, $\,\a\,$ and $\,\b\,$ have embedded 
	lifts. If they \tit{meet} above $\,\b(b)\,$ for some $\,b\in\US^{1}$, then
	Lemma \ref{lem:repo} provides an $\,a\in\US^{1}$ and a $\,\d>0\,$ 
	such that $\,(a-\d,a+\d)\,$ and $\,(b-\d,b+\d)\,$ lift, via $\,\a\,$ and $\,\b\,$
	respectively, to the same arc in $\,\R^{2}\times\US^{1}$. The lifts of 
	$\,\a\,$ and $\,\b\,$ therefore meet along a set relatively open in the image of each. 
	The coincidence set is also closed (trivially) so the two lifts coincide entirely, 
	manifesting (via the Inverse Function Theorem) a $\,C^{1}\,$ transition diffeomorphism 
	between them.
	
	The identity map on $\,\R^{2}\times\US^{1}$ then induces a diffeomorphism between the 
	circles parametrizing $\,\a\,$ and $\,\b$, allowing us to read $\,\b\,$ as a 
	reparametrization of $\,\a$. Orientation is preserved, for the lifts would otherwise be
	completely disjoint by \hyperref[obs:lift]{Observation \ref{obs:lift}}. 
	
	If the lifts \tit{are} completely disjoint, then (since clean immersions have at 
	most finitely many double-points by \hyperref[lem:clean]{Lemma \ref{lem:clean}}) 
	we can find a point $\,p\in\a(\US^{1})\,$ with a single pre-image $\,\{t\}=\a^{-1}(p)$. 
	Then $\,\a\,$ and $\,\b\,$ share a unique tangent line at $\,p$. If their lifts don't 
	meet, $\,\b\,$ must lift to $\,(p,-\dot\a(t_0))\,$ above $\,p$. The lift of any 
	orientation-\tit{reversing} reparametrization  $\,\b'\,$ of $\,\b\,$ thus \tit{meets} 
	that of $\,\a\,$ above $\,p$, making $\,\b'\,$ an orientation-\tit{preserving} 
	reparametrization of $\,\a\,$ by what we have already proven. So $\,\b\,$ 
	\tit{reverses} orientation, as claimed.
\end{proof}

We will apply the \hyperref[prop:repo]{Proposition} just proven mainly via this
immediate

\begin{cor}\label{cor:repo}
	Any clean \emph{central} $\,C^{2}\,$ loop is reparametrized by its central symmetry.
\end{cor}

As \hyperref[fig:egO8]{Figure \ref{fig:egO8}} shows, the reparametrization induced by a 
central symmetry of a clean loop may preserve or reverse orientation. The two 
possibilities have starkly different geometric implications, however. To see that, 
we will need \hyperref[cor:cm]{Corollary \ref{cor:cm}} below---a further consequence 
of \hyperref[prop:repo]{Proposition \ref{prop:repo}}---which requires this

\begin{definition}\label{def:cm}
	The \tit{centroid} ({center of mass}) $\,\mu(\a)\,$ of a $\,C^{1}\,$ loop 
	$\,\a\colon\US^{1}\to\R^{n+1}\,$ with length $\,L\,$ is the mean value of $\,\a\,$ 
	relative to an arc-length parameter $\,s\,$:
	\[
		\mu(\a):=\frac{1}{L}\int_{\US^{1}}\a(s)\ ds
	\]
\end{definition}

Note that \tit{the centroid of a loop with central symmetry may not coincide with its 
center of symmetry.} For example, take the circles $\,(x\pm 1)^{2}+y^{2}=1$, and 
parametrize their union, starting at $\,\nv 0$, by tracing clockwise around the 
right-hand lobe, then counterclockwise around the left, and finally, clockwise around 
the right again. The origin will be a center of symmetry, but the centroid lies at $\,(1/3,0)$. 
Clean loops, however, never exhibit this kind of discrepancy:

\begin{cor}\label{cor:cm}
	For a clean, central $\,C^{2}\,$ loop, center of symmetry and centroid coincide. 
\end{cor}

\begin{proof}
	Suppose $\,\a\co\US^{1}\to\R^{2}$ is a clean $\,C^{2}\,$ loop with center of symmetry 
	at $\,\nv c\in\R^{2}$, and length $\,L$. View it as a unit-speed $L$-periodic 
	immersion $\,\R\to\R^{2}$. \hyperref[prop:repo]{Proposition \ref{prop:repo}} provides 
	a diffeomorphism $\,\phi\,$ of the circle (which lifts to $\,\R\,$) such that 
	$\,2\nv c-\a=\a\circ\phi$. By the chain rule and constancy of 	speed (which is 
	preserved by the reflection), we must also have $\,|\phi'|\equiv 1$. If we denote 
	the unit-speed parameter for $\,\a\,$ by $\,s$, then $\,u=\phi(s)\,$ gives a 
	unit-speed parameter for its reflection $\,2\nv c-\a\,$, whose centroid is then 
	clearly\smallskip
	\begin{align*}
		2\nv c-\mu(\a)
		&=
		\frac{1}{L}\int_{0}^{L}2\nv c-\a(s)\ ds\\
		&=
		\frac{1}{L}\int_{0}^{L}\a\circ\phi(s)\ ds\\
		&=
		\frac{1}{L}\int_{0}^{L}\a\circ\phi(s)\,|\phi'(s)|\ ds\\
		&=
		\frac{1}{L}\int_{0}^{L}\a(u)\ du\\
		&=\mu(\a)
	\end{align*}
	Thus $\,\mu(\a)=\nv c$, as claimed.
\end{proof}
\medskip

\begin{definition}\label{def:HLC}
	When an immersed $\,C^{1}$\, loop $\,\a\colon\US^{1}\to\R^{2}$ is central with respect
	to $\,\nv c\in\R^{2}$, we call the line 
	segment joining $\,\a(t)\,$ to $\,2\nv c-\a(t)\,$ a \tit{diameter} of $\,\a$. If we can
	parametrize $\,\a\,$ so that\smallskip
	\begin{equation}\label{eqn:dc}
		2\nv c-\a(t) = \a(t+\pi)
	\end{equation}
	for all $\,t\in\US^{1}$ (intertwining reflection through 
	$\,\nv c\,$ with the antipodal map on $\,\US^{1}$) we say that $\,\a\,$ is 
	\tit{diameter-central}.
		
	diameter-central loops are obviously central, but the converse is false, as shown by the 
	central figure-8 in \hyperref[fig:egO8]{Figure \ref{fig:egO8}}. Careful
	consideration of that picture reveals that the figure-8 is not diameter-central. 
\eob
\end{definition}

\goodbreak
\begin{prop}\label{prop:oo0}
	Suppose $\,\a\colon\US^{1}\to\R^{2}$ is a clean, central $\,C^{2}\,$ loop. Then either
	\begin{itemize}
	
		\item[a)]
		the symmetry preserves orientation, in which case $\,\a\,$\smallskip
		\begin{itemize}
			\item
			is regularly homotopic to $\,e^{(2k+1)\th}\,$ for some $\,k\in\Z$\smallskip
			
			\item
			avoids its center, and\smallskip
			
			\item
			is diameter-central.\smallskip

		\end{itemize}
		
		\item[]
		or\medskip
		\goodbreak
				
		\item[b)]
		the symmetry reverses orientation, in which case $\,\a\,$\smallskip
		\begin{itemize}
			\item
			is regularly homotopic to the figure-8,\smallskip
			
			\item
			has a simple double-point at its center, and\smallskip
			
			\item
			is \tit{not} diameter-central.\smallskip

		\end{itemize}
	\end{itemize}
\end{prop}

\begin{proof}
	We can assume $\,\a\,$ is centered at the origin $\,\nv 0$, and (after a homothety
	giving it length $\,2\pi\,$) has unit speed.
	\hyperref[cor:repo]{Corollary \ref{cor:repo}} then gives $\,-\a = \a\circ\phi\,$ 
	for some diffeomorphism  $\,\phi\colon\US^{1}\to\US^{1}$. By the chain rule, our 
	unit speed assumption forces $\,|\phi'|\equiv 1$, making $\,\phi\,$ an 
	\tit{isometry} of $\,\US^{1}$. An isometry either {rotates} $\,\US^{1}$ or {reflects 
	it across a diameter}, preserving or reversing orientation respectively.
	
	When we view $\,\US^{1}\approx\R/2\pi\,$ as an additive group, rotation takes the form 
	$\,\phi(t)= t+l\,$ for some $\,l\in\US^{1}$. So if the symmetry preserves orientation, 
	we get $\,-\a(t)=\a(t+l)\,$ for all $\,t$. Since $\,\a\,$
	is not constant, we may assume $\,0<|l|\le\pi$. Iterating the symmetry then gives
	$\,\a(t+2l)=\a(t)$, and hence $\,\dot\a(t+2l)=\dot\a(t)$. Having
	clean double points, however, obstructs this pair of identities for any $\,0<|l|<\pi$. 
	So in the orientation-preserving case, we must have $\,|l|=\pi$, which 
	makes $\,\a\,$ diameter-central, as conclusion (a) asserts. 
	
	A diameter-central loop has parallel tangent lines at $\,\a(t+\pi)\,$ and $\,\a(t)$, as 
	follows from differentiating (\ref{eqn:dc}). Since we assume clean double-points, 
	this forces $\,\a(t+\pi)\ne\a(t)\,$ for all $\,t\in\US^{1}$. But we just saw that 
	$\,\a(t+\pi)=-\a(t)\,$ for all $\,t\in\US^{1}$. So in the orientation-preserving case, 
	our loop must avoid the origin---its center---as claimed by (a).
	
	In the orientation-\tit{reversing} case, by contrast, $\,\a\,$ is reparametrized by an
	isometry $\,\phi\colon\US^{1}\to\US^{1}$ that reflects across a diameter, fixing two 
	antipodal points that we can assume, after a rotation, to be $\,t=0\,$ and $\,t=\pi$. 
	In this case, for all $\,t\in\US^{1}$, we have $\,\phi(t)=-t$, and thus\smallskip
	\begin{equation}\label{eqn:8c}
		-\a(t)=\a(-t)
	\end{equation}
	A central symmetry fixes \tit{only} its center, however, forcing $\,\a\,$ to
	map both $\,t=0\,$ and $\,t=\pi\,$ to the origin. In fact, the origin must be a 
	\tit{simple} double-point, as (b) claims. For, any central loop has parallel tangent 
	lines at the ends of diameters, and when (\ref{eqn:8c}) holds, that means parallel 
	tangent lines at $\,\a(t)\,$ and $\,\a(-t)\,$ for every $\,t\in\US^{1}$. If we had 
	$\,\a(t)=\a(-t)\,$ for some $\,t\,$ \tit{not} fixed by $\,\phi$, we would breach our 
	clean double-points assumption.
	
	It remains to verify the claims about regular homotopy. As is well-known, (e.g.,
	\cite{whitney} or \cite[Proposition 2.1.6]{kl})
	the regular homotopy class of an immersed plane loop $\,\a\colon\US^{1}\to\R^{2}$ is 
	classified by its \tit{rotation index}---the degree $\,\w_{\a}\in\Z\,$ of its unit 
	tangent map $\,\a'/|\a'|\colon\US^{1}\to\US^{1}$, which we may compute by 
	integrating the geodesic curvature \hyperref[eqn:kg]{(\ref{eqn:kg})} along 
	$\,\a$:\smallskip
	\begin{equation}\label{eqn:tnum}
		\w_{\a}=\frac{1}{2\pi}\int_{0}^{2\pi}\k_{g}(t)\ dt,
	\end{equation}
			
	Consider first the orientation-preserving case. There, as we have seen, $\,\a\,$ is 
	diameter-central: $\,\a(t+\pi)=-\a(t)\,$ for all $\,t$. It follows trivially
	that velocity and acceleration change sign too when we rotate the input by $\,\pi$. 
	As easily seen from formula \hyperref[eqn:kg]{(\ref{eqn:kg})}, however, this makes 
	$\,\k_{g}\,$ \tit{even} on the circle: $\,\k_{g}(t+\pi)=\k_{g}(t)$. So when orientation 
	is preserved, the total signed curvature of $\,\a\,$ is twice that along the arc 
	$\,\a(0,\pi)$. At the same time, we have $\,\dot\a(\pi)=-\dot\a(0)$, forcing the unit 
	tangent $\,\dot\a/|\dot\a|\,$ to traverse an odd number of semicircles as $\,t\,$ varies 
	from $\,0\,$ to $\,\pi$. So\smallskip
	\[
		\int_{0}^{2\pi}\k_{g}(t)\ dt
		=
		2\int_{0}^{\pi}\k_{g}(t)\ dt 
		= 2(2k+1)\pi\quad\text{for some $\,k\in\Z$}.
	\]
	
	By (\ref{eqn:tnum}), we then have $\,\w_{\a}=2k+1$, an odd integer, as claimed.

	In the orientation-\tit{reversing} case, identity (\ref{eqn:8c}) replaces the 
	diameter-central condition above. Differentiate that identity twice and use 
	(\ref{eqn:tnum}) to see that $\,\k\,$ is now an \tit{odd} function on the circle:
	\[
		\kappa(-t)=-\kappa(t)\,.
	\]
	The integral of an odd function vanishes, so (\ref{eqn:tnum}) now yields $\,\w_{\a}=0$, 
	making $\,\a\,$ regularly homotopic to the figure-8, as stated. This completes our 
	proof.
\end{proof}

\begin{cor}
	A clean $\,C^{2}\,$ plane loop with \emph{even, non-zero} rotation index \emph{cannot} 
	have central symmetry.
\end{cor}
\bigskip

%%%%% 
\section{Main Result}\label{sec:MR}
Supposing $\,M^{n}\,$ is a smooth manifold, we now take up our motivating question: 
\tit{What can we say about a complete, proper immersion $\,F\:M^{n}\to\R^{n+1}$ when
$\,F(M)\,$  has central intersections with an open set of hyperplanes?}

To address this, we introduce some notation.
We write $\,u_{p}^{\bot}\,$ for the hyperplane containing $\,p\in\R^{n+1}$ and normal to
$\,u\in\US^{n}$. When $\,p\,$ is the origin, we simply write $\,u^{\bot}$.
These hyperplanes are, respectively, zero sets of the affine functions $\,u_{p}^{*}\,$ 
and $\,u^{*}\,$ given by\smallskip
\[
	u_{p}^{*}(x)=u\cdot (x-p),\quad u^{*}(x) = u\cdot x
\]
When using this notation, we always assume $u\,$ to be a \tit{unit} vector. 
We denote the angular distance between unit vectors $\,u, v\in\US^{n}\,$ by 
$\phi(u,v):=\arccos(u\cdot v)$.

When $\,a>0\,$ and $\,P=u_{p}^{\bot}$, we write $\,P_{a}\,$ for the $a$-neighborhood of 
the hyperplane $\,P\,$:\begin{equation}\label{eqn:Pa}
	P_{a}:=\l\{q\in\R^{n+1}\: \l|u_{p}^{*}(q)\r|<a\r\}
\end{equation}

We call $\,\nu\in\US^{n}\,$  a \tit{unit normal} to an immersion $\,F\:M^{n}\to\R^{n+1}\,$
at a point $\,x\in M\,$ if $\,\nu\,$ is orthogonal to the hyperplane $\,dF(T_{x}M)\,$ in $\,\R^{n+1}$.

We can then say that $\,F\,$ has \tit{general position} if, whenever
$\,y\in\R^{n+1}\,$ and $\,\nu_{1}, \nu_{2},\dots, \nu_{k}\,$ are unit normals 
to $\,F\,$ at distinct points in $\,F^{-1}(y)$, we have\smallskip
\begin{equation}\label{eqn:gp}
	\nu_{1}\wedge\nu_{2}\wedge\cdots\wedge\nu_{k}\ne 0
\end{equation}

If this holds when we extend $\,F\,$ to $\,M\cup P\,$ for some hyperplane $\,P\subset\R^{n+1}$
via the inclusion map on $\,P$, we say that \tit{$\,F\,$ and $\,P\,$ are in general position}.
Note that in this case, the restriction of $\,F\,$ to $\,M\,$ must itself have general position.

When \hyperref[eqn:gp]{(\ref{eqn:gp})} holds for $\,k=2\,$ (i.e., whenever $\,\nu_{1},\nu_{2}\,$
are unit normals to $\,F\,$ at distinct points of $\,F^{-1}(y)\,$), we get weaker 
conditions that we respectively express by saying \tit{$\,F\,$ has transverse self-intersections}, or 
\tit{$\,P\,$ meets $\,F\,$ transversally}.

Transversality alone makes $\,F^{-1}(P)\,$ an embedded
hypersurface in $\,M$ \cite[p.22]{hirsch}. General position guarantees more:
when $\,n=2$, for instance, it is not hard to see that it makes all double-points of 
$\,P\cap F(M)\,$ \tit{clean} as specified in \hyperref[def:cdp]{Definition \ref{def:cdp}}.

We want to focus on the case where $\,F\,$ and $\,P\,$ have general position and the 
\tit{compact} components of $\,F^{-1}(P)\,$ map to sets with central symmetry. Two definitions 
will help:

\begin{definition}[Cross-cut]\label{def:cc}
	When a hyperplane $\,P\subset\R^{n+1}$ meets an immersion $\,F\:M^{n}\to\R^{n+1}$ 
	transversally, a \tit{cross-cut of $\,F\,$
	relative to $\,P\,$} is a {compact component} $\,\G\subset F^{-1}(P)$. We also call 
	its image $\,F(\G)\,$ a cross-cut; context will signal which meaning applies.
	
	We call $\,\G\,$ a \tit{clean} cross-cut when $\,P\,$ and $\,F\,$ are in general 
	position.
\eob
\end{definition}

The transversality assumption in \hyperref[def:cc]{Definition \ref{def:cc}} ensures that 
the tangential projection $\,u\mapsto u-(u\cdot\nu)\nu\,$ yields a \tit{non-vanishing}
transverse vectorfield along $\,\G\,$ (the choice of unit normal $\,\nu\,$ to $\,F\,$ is obviously 
irrelevant here). Cross-cuts are thus \tit{orientable} in $\,M\,$. A routine differential 
topology exercise then yields the existence of what we shall call a \tit{good tubular 
coordinate neighborhood} $\,U\,$ of a cross-cut $\,\G$. This is a neighborhood that $\,F\,$ 
maps to a tube foliated by cross-cuts diffeomorphic to $\,\G$, each a level set of the height 
function $\,u_{p}^{*}\,$.

\begin{definition}[Good tubular patch]\label{def:gt}
	Suppose $\,\G\subset M\,$ is a cross-cut of $\,F\,$ relative to a hyperplane 
	$\,P=u_{p}^{\bot}$. By a \tit{good tubular coordinate neighborhood}
	(or \tit{good tubular patch}) for $\,\G$, we mean a pair $\,(U,\psi)$, where $\,U\subset M\,$
	is the image of an embedding $\,\psi\:\G\times[-a,a]\to M\,$ for some $\,a>0$, 
	and $\,\psi\,$ has these three properties for all $\,(\th,h)\in\G\times[-a,a]\,$:
	
	\begin{itemize}
		
		\item[a)]
		$\psi(x,0)=x\,$ for all $\,x\in\G\,$\medskip
		
		\item[b)]
		$\l(u_{p}^{*}\circ F\circ \psi\r)(\th,h) = h\,$ and \medskip
		
		\item[c)]
		$d\l(u_{p}^{*}\circ F\circ \psi\r) \ne 0\,$
	\end{itemize}
	
	Property (b) means that for each $\,h\in[-a,a]$, the composition $\,F\circ\psi\,$
	maps $\,\G\times\{h\}\,$ into the plane $\,u_{p}^{*}\equiv h\,$.
	Property (c) makes $\,F\,$ transverse to these same planes, so that
	$\,\psi(\G\times\{h\})\,$ is a cross-cut of $\,F\,$ for each $\,h\in[-a,a]\,$. 
\eob
\end{definition}

As mentioned above, the \tit{existence} of a good tubular neighborhood of a cross-cut
is easy to establish. When a cross-cut is \tit{clean}, we can guarantee that nearby 
cross-cuts are likewise clean: 

\begin{lem}\label{lem:gt}
	Suppose $\,\G\subset M\,$ is a clean cross-cut relative to $\,P=u_{p}^{\bot}$, and 
	$\,(U,\psi)\,$ is a good tubular patch for $\,\G\,$. Then there is an $\,\eps>0\,$ 
	for which $\,|q-p|<\eps\,$ and $\,\phi(v,u)<\eps\,$ together ensure that 
	$\,F^{-1}(v_{q}^{\bot})\cap U\,$ is again a clean cross-cut, and is regularly homotopic 
	to $\,\G\,$.
\end{lem}

\begin{proof} 
	Define the map\smallskip
	\[
		\mathcal{F}\:U\times\R^{n+1}\times\US^{n}\to\R\times\R^{n+1}\times\US^{n}
	\]
	via\smallskip
	\[
		\mathcal{F}(x,q,v)=\l((F(x)-q)\cdot v,\,q,\,v\r)
	\]
	
	Property (c) in \hyperref[def:gt]{Definition \ref{def:gt}} makes $\,d\mathcal{F}\,$
	surjective at each point of $\,\mathcal{F}^{-1}(0,p,u)=\G\times\{p\}\times\{u\}$, and
	lower-semicontinuity of rank then makes $\,d\mathcal{F}\,$ surjective on some 
	neighborhood of $\,\G\times p\times u\,$. If we denote $\,\eps$-neighbor\-hoods of 
	$\,p\,$ and $\,u\,$ in $\,\R^{n+1}\,$ and $\,\US^{n}$ respectively by $\,B_{\eps}(p)\,$ 
	and $\,B_{\eps}(u)\,$, the Implicit Function Theorem and compactness of $\,\G\,$ then 
	make it straightforward to deduce that for some $\,\eps>0$, the 
	$\mathcal{F}$-preimage of $\,(-\eps,\eps)\times B_{\eps}(p)\times B_{\eps}(u)\,$ is
	foliated by preimages $\,\mathcal{F}^{-1}(h,q,v)$, all regularly homotopic to 
	$\,\mathcal{F}^{-1}(0,p,u)=\G\times\{p\}\times\{u\}$. It follows that 
	$\,\mathcal{F}^{-1}(\{0\}\times B_{\eps}(p)\times B_{\eps}(u))\,$ is likewise 
	foliated. Since\smallskip
	\[
		\mathcal{F}^{-1}(0,q,v)=\l(F^{-1}\l(v_{q}^{\bot}\r)\cap U\r)\times \{q\}\times \{v\}
	\]
	
	this shows that $\,|q-p|<\eps\,$ and $\,\phi(v,u)<\eps\,$ together ensure, 
	for every such $\,q\,$ and $\,v$, that $\,F^{-1}(v_{q}^{\bot}) \cap U\,$ is a cross-cut 
	regularly homotopic to $\,\G\,$.

	Finally, by making $\,\eps>0\,$ smaller still if necessary, we can guarantee that 
	these cross-cuts are all \tit{clean} too. Otherwise, we could find convergent sequences 
	$\,(q_{k})\to p\,$ and $\,(v_{k})\to u\,$ for which each corresponding cross-cut
	$\,(F(x)-v_k)\cdot{q_{k}}\equiv 0\,$ in $\,U\,$ was \tit{not} clean. 
	Condition \hyperref[eqn:gp]{(\ref{eqn:gp})} would then have to fail at some point
	$\,y_{k}\,$ in each of these cross-cuts. Condition \hyperref[eqn:gp]{(\ref{eqn:gp})} 
	is continuous in all variables, however, $\,F\,$ is $\,C^{1}$, and $\,U\,$ is
	compact. Passing to a subsequence, we could then take a limit as $\,k\to\infty\,$ 
	and force a contradiction to our assumption that $\,\G\,$ itself was clean.
\end{proof}

With these purely differential-topological facts in hand, we now turn the case
of interest: where (the images of) all cross-cuts have central symmetry.

\begin{definition}[\cx]
	An immersion $\,F\colon M\to\R^{n+1}\,$ has the \tit{central cross-cut property} 
	(abbreviated \cx) when
	
	\begin{itemize}
		
		\item[a)]
		At least one clean cross-cut exists, and\medskip

		\item[b)]
		The image of every clean cross-cut has \tit{central symmetry}. 
	\end{itemize}
Note that \cx\ is an {affine}-invariant property: if $\,F\,$ has \cx, and 
$\,A\,$ is an affine isomorphism of $\,\R^{n+1}$, then $\,A\circ F\,$ has \cx\ too.
\eob
\end{definition}

In $\,\R^{3}$, circular cylinders and spheres have \cx, and they represent the only
two kinds of examples we know:

\begin{itemize}

	\item[---]\hypertarget{tgt:ccyl}{\tit{Central cylinders:}}
	If an immersion with a cross-cut is preserved by a line of translations \tit{and} 
	by a central reflection, we call it a \tit{central 
	cylinder}. Central cylinders clearly have \cx, since every cross-cut is a
	translate of one through the center.\medskip
	
	\item[---]\hypertarget{tgt:quad}{\tit{Tubular quadrics:}}
	When a non-degenerate {quadric} hypersurface in $\,\R^{n+1}$ is affinely equivalent to 
	a locus of the form
	\[
		x_{1}^{2}+x_{2}^{2}+\cdots x_{n}^{2}\pm x_{n+1}^{2}=c\in\R
	\]
	it will always have compact and transverse, hence ellipsoidal (and thus central)
	cross-cuts. We call these hypersurfaces \tit{tubular quad\-rics}. Note that in 
	$\,\R^{3}$, \tit{all} non-degenerate quadrics are tubular.
\end{itemize}

We suspect these two classes exhaust all possibilities: 

\begin{con}\label{con:gen}
	A complete immersion $\,F\colon M^{n}\to\R^{n+1}\,$ with \cx\ must either be
	a central cylinder, or a tubular quadric.
\end{con}

In previous papers, we confirmed weakened versions of this conjecture, proving it
\begin{itemize}

	\item[---]
	for $\,C^{1}\,$ hypersurfaces of revolution ($\,\mathrm{SO}(n)\,$ symmetry) in $\,\R^{n+1}$
	\cite{sor}, and then, using that result,\medskip
	
	\item[---]
	for $\,C^{2}\,$ surfaces in $\,\R^{3}\,$ whose cross-cuts 
	are \tit{convex} as well as central \cite{cpo}.
\end{itemize}

Here we add another case to this list: roughly, that of a complete surface in $\,\R^{3}$ with 
\cx\ \tit{and} for which some clean cross-cut is a figure-8.

To make this precise, we first note that on any complete immersed $\,C^{2}\,$ surface 
with \cx\ in $\,\R^{3}$, every clean cross-cut is a (clean) central $\,C^{2}\,$
plane loop. By \hyperref[prop:oo0]{Proposition \ref{prop:oo0}}, then, each of these
loops is either regularly homotopic to a figure-8, \tit{or} has odd {rotation index}. 

The rotation index of a figure-8 is zero, and here (as sketched in our introduction) we
verify \hyperref[con:gen]{Conjecture \ref{con:gen}} for immersions with figure-8 cross-cuts. 
Since cross-cuts of quadrics can't be figure-8's, such immersions must be cylindrical:

\begin{thm}[Main Result]\label{thm:8case}
	If $\,F\colon M\to\R^{3}\,$ is a complete $C^{2}$ immersion with \cx, 
	and some plane in general position relative to $\,F\,$ cuts it along a clean figure-8, 
	then $F(M)$ is a central cylinder.
\end{thm}

The figure-8 assumption is decisive for the following reason: When a plane $\,P\,$ 
cuts a surface with \cx\ transversally along a figure-8 centered at $\,c\in\R^{3}$, 
and we tilt $\,P\,$ slightly about $\,c\,$ to get nearby cross-cuts, \tit{the 
latter remain centered at $\,c$}. 

Without the figure-8 assumption, this fails. 

Indeed, consider the unit sphere $\,\US^{2}\subset\R^{3}$. It clearly has \cx.
Now take $\,u\in\US^{2}$, $\,0<\lam<1$, and set $\,c:=\lam u\,$. The plane 
$\,u_{c}^{\bot}\,$ will cut $\,\US^{2}\,$ along a circle centered at $\,c$. For any 
$\,v\in\US^{2}\,$ near $\,u$, however, the cross-cut $\,v_{c}^{\bot}\cap\US^{2}\,$ is 
clearly centered on the line spanned by $\,\lam v\,$ (\hyperref[fig:cmoves]
{Figure \ref{fig:cmoves}}). So for $\,v\ne u$, the center moves.

	\begin{figure}[ht]
		\includegraphics[height=48mm]{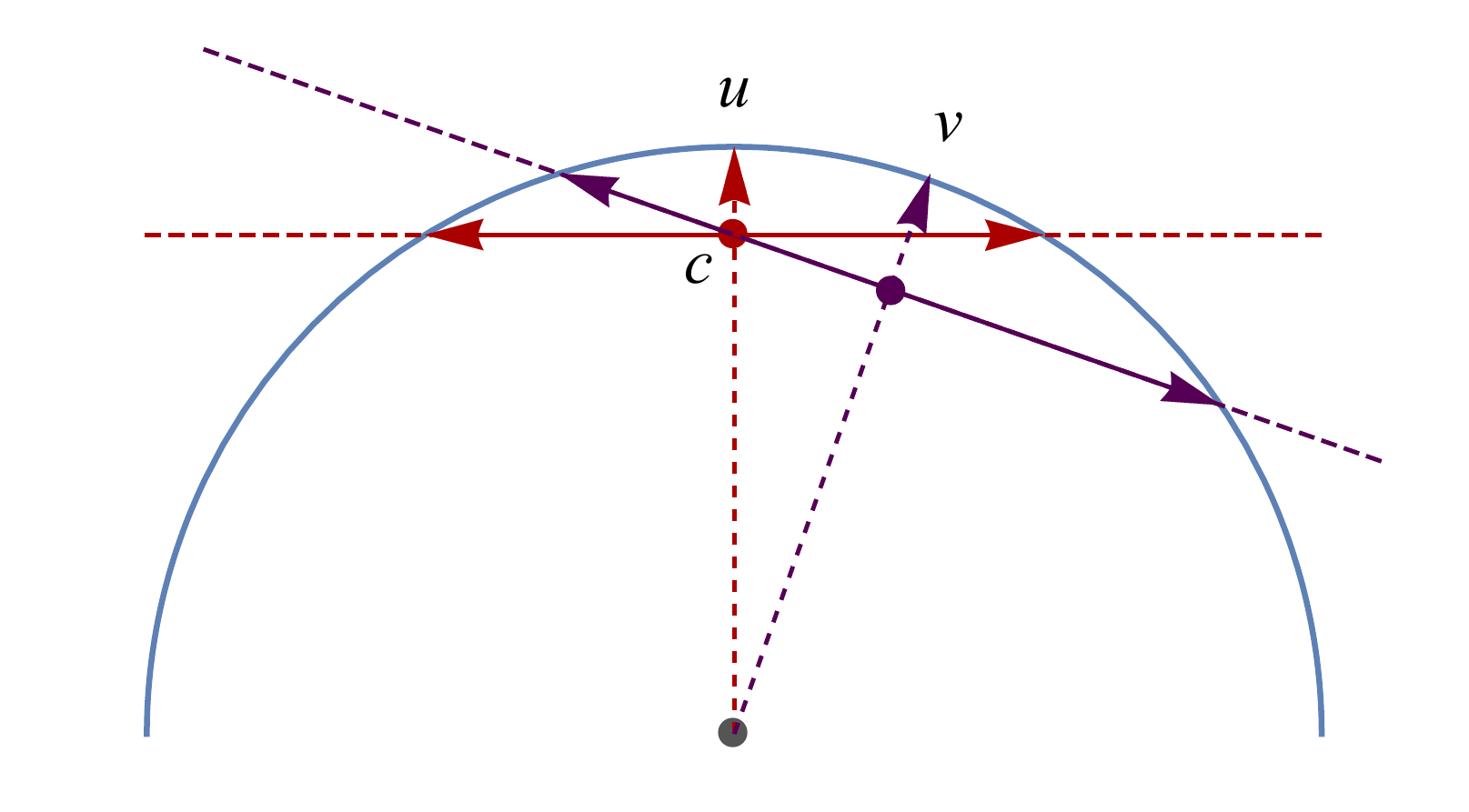}
		\caption{Cross-cuts on a sphere, via $\,u_{c}^{\bot}\,$ and $\,v_{c}^{\bot}$.
		Both hyperplanes \tit{contain} $\,c$, but only one of the cross-cuts 
		(red) is {centered} at $\,c\,$.}
		\label{fig:cmoves}
	\end{figure}

The center \tit{cannot} move in this way when cross-cuts are figure-8's, as we 
make precise in \hyperref[lem:cdnm]{Lemma \ref{lem:cdnm}} shortly below, using the 
notion of \tit{central curve} of a good tubular patch:

\begin{definition}[Central curve]\label{def:ccv}
	Suppose $\,\G\,$ is a cross-cut for an immersion $\,F\:M\to\R^{n+1}\,$ 
	relative to a hyperplane $\,P=u_{p}^{\bot}$. Let $\,(U,\psi)\,$ denote a good tubular patch 
	for $\,\G\,$ as in \hyperref[def:gt]{Definition \ref{def:gt}},
	so that $\,F\,$ maps $\,\psi(\G,h)$ into the hyperplane $\,u_{c}^{*}\equiv h\,$
	for each $\,h\in[-a,a]$. The \tit{central curve} of the patch is the map 
	$\,\c\:[-a,a]\to\R^{n+1}\,$ sending any $\,h\in[-a,a]\,$ to the centroid 
	$\,\c(h)\,$ of $\,F(\psi(\G,h))\,$.
\eob
\end{definition}

When $\,F\,$ is $\,C^{k}$, the central curve of a good tubular patch is
clearly $\,C^{k}$ too. It is also \tit{immersed}, since condition (b) from \hyperref[def:gt]
{Definition \ref{def:gt}} yields $\,u_{p}^{*}(\c(h))=h$, and hence $\,\dot\c(h)\cdot u\ge 1$.

In the Lemma below, we formulate the advantage offered by figure-8 cross-cuts.
Notation is as above: $F\colon M^{2}\to\R^{3}$ is a 
proper $\,C^{2}\,$ immersion, $\,u\in\US^{2}$ and $\,c\in\R^{3}$ are fixed. We have a
clean cross-cut $\,\G\subset F^{-1}(u_{c}^{\bot})$, for which $\,U\subset M\,$ is a good tubular 
neighborhood (\hyperref[def:gt]{Definition \ref{def:gt}}), and $\,\c\:[-a,a]\to\R^{3}$ is
its central curve.

\begin{lem}\label{lem:cdnm} 
	Suppose $\,F\,$ has \cx, $\,F(\G)\,$ is a figure-8, and $\,\eps>0\,$. If 
	$\,\G_{h,v}:=U\cap F^{-1}(v_{\c(h)}^{\bot})\,$ is a clean cross-cut, regularly homotopic to 
	$\,\G\,$ whenever $\,|h|<\eps$ and $\,v\in\US^{2}$ with $\,\phi(v,u)<\eps$, then 
	$\,F(\G_{h,v})$ is a figure-8 with central symmetry about $\,\c(h)\,$ for all such $\,h\,$ and $\,v$.
\end{lem}

\begin{proof}
	\hyperref[lem:gt]{Lemma \ref{lem:gt}} says that for all sufficiently small $\,|h|$, 
	the cross-cut $\,\G_{h,u}\,$ (cut by the plane at signed height $\,h\,$ above 
	$\,u_{c}^{\bot}$) is,
	like $\,\G_{0,u}=\G\,$ itself, clean and regularly homotopic to $\,\G\,$. For simplicity, 
	we can assume this holds for all $\,|h|\le a\,$. (If not, re-define our good tubular patch 
	using a smaller $\,a>0$.)
	
	In this case, $\,F(\G_{h,u})\,$ is a clean figure-8 for every $\,|h|\le a$, and its center, 
	by \hyperref[prop:oo0]{Proposition \ref{prop:oo0}(b)}, is a simple double-point. The
	central curve $\,\c\,$ of the patch thus consists entirely of simple double-points.
	
	In particular, if we fix any $\,h\in(-\eps,\eps)$, then $\,F^{-1}(\c(h))\cap U\,$ is a {pair} 
	$\,\{x_{1},x_{2}\}$, and as an immersion, $\,F\,$ \tit{embeds} disjoint 
	neighborhoods $\,U_{1}\supset x_{1}\,$ and $\,U_{2}\supset x_{2}\,$ in such a way that, 
	in the ball $\,B_{h,r}\,$ centered at $\,\c(h)\,$ with sufficiently small radius
	$\,r>0\,$, we have
	\[
		 F(U)\cap B_{r,h}\subset F(U_{1})\cup F(U_{2})
	\]

	Further, since $\,F\,$ has general position, we can make $\,r>0\,$ small enough to 
	ensure that in $\,B_{r,h}$, the sheets $\,F(U_{1})\,$ and $\,F(U_{2})\,$ meet 
	along a segment of the central curve and nowhere else.
	
	Now, as long as $\,\phi(v,u)<\eps$, the nearby cross-cut $\,\G_{h,v}\,$ is, by assumption,
	another clean cross-cut in $\,U$, regularly homotopic to $\,\G=\G_{0,u}$.  
	Immersion preserves regular homotopy, so for all such $\,v\,$, the nearby images 
	$\,F(\G_{h,v})$, like $\,F(\G)$, are clean figure-8's---and they are central, since 
	$\,F\,$ has \cx. We just need to show they stay centered, like $\,F(\G_{h,u})$, 
	at $\,\c(h)$.
	
	To see that they are, note that the planes $\,v_{\c(h)}^{\bot}\,$ all cut the 
	central curve $\,\c\,$ transversally at $\,\c(h)\,$ since the cross-cuts they 
	form are all clean. So by shrinking $\,r>0\,$ further if needed, we can 
	ensure that in $\,B_{r,h}\,$, each of these planes cuts the central curve \tit{only} at 
	$\,\c(h)\,$. 
	
	It follows that $\,\c(h)\,$ is the \tit{unique} double-point that $\,F(\G_{h,v})\,$ 
	has in $\,B_{r,h}\,$. Since $\,\G_{h,v}\,$ varies smoothly with 
	$\,v$, its centroid---and center of symmetry by \hyperref[cor:cm]{Corollary 
	\ref{cor:cm}}---varies smoothly too. So for $\,v\,$ sufficiently near $\,u$, 
	the center of the figure-8 $\,F(\G_{h,v})\,$ must stay in $\,B_{r,h}$. As seen above, 
	however, that center is a simple double-point, and we have just noted that for every
	$\,v\,$ in question, the \tit{only} double-point of $\,F(\G_{h,v})\,$ in $\,B_{r,h}\,$ is 
	$\,\c(h)$. When $\,\phi(v,u)<\eps$, the center of $\,F(\G_{h,v})\,$ is therefore trapped at 
	$\,\c(h)\,$. As this holds for whenever $\,|h|<\eps$, we have proven the 
	Lemma.
\end{proof}

We will prove our main result (\hyperref[{thm:8case}]{Theorem \ref{thm:8case}}) by
combining this {Lemma} with the \hyperref[lem:axis]{Local Axis Lemma} below, 
which shows that when $\,F\,$ has \cx, and centers of tilted cross-cuts stay (locally)
on the central curve as in the Lemma above, the central curve is locally straight. 
Note that it makes no figure-8 assumption. This lemma quickly produces 
a local version of the main result, namely \hyperref[cor:locCyl]{Corollary \ref{cor:locCyl}}.

As above, we write $\,a>0\,$ and $\,P=u_{c}^{\bot}\,$ for a fixed (but arbitrary)
scalar and plane respectively; $\,P_{a}\,$ denotes the $a$-neighborhood of $\,P\,$. 
We have a clean cross-cut $\,\G\subset F^{-1}(P)\,$, and  a good tubular patch 
$\,\psi\:\G\times[-a,a]\to U\subset M\,$ around $\,\G\,$, so that 
$\,F(\partial U)\subset\bd P_{a}\,$. Without loss of generality, we assume $\,c=\c(0)$, 
the initial value of the central curve $\c\,$ of $\,F(U)\,$.

\begin{lem}[Local axis lemma]\label{lem:axis} 
	Suppose $\,\eps>0\,$, $\,0<b<a\,$ and $\,F^{-1}(v_{\c(t)}^{\bot})\cap U\,$ is a 
	boundaryless clean cross-cut whose image is central about $\,\c(t)$ whenever
	$\,\phi(u,v)<\eps\,$ and $\,|t|<b\,$. Then $\,\c\,$ maps $\,[-b,b\,]\,$ to a line segment.
\end{lem} 

\begin{proof} 
	We may identify $\,\G\approx\US^{1}$, and 
	simplify notation accordingly by using coordinates from the domain of our good tubular 
	patch so that, for instance, $\,F(\th,h)\,$ really means $\,F(\psi(\th,h))\,$. 
	
	Fix an arbitrary $\,\z\in(-b,b)$, and choose $\,\th_{0}\in\US^{1}$ so that 
	\tit{$\,p_{0}:=F(\th_{0},\z)\,$ maximizes $\,|F(\th,\z)|^{2}\,$ on 
	$\,F(\G,\z)\,$}:\smallskip
	\[
		\l|p_{0}\r|^{2}=\l|F(\th_{0},\z)\r|^{2}\ge \l|F(\th,\z)\r|^{2}\quad\text{for all $\,\th\in\US^{1}$}
	\]
	
	To prove the Lemma, we will first need to show that $\,(\th_{0},\z)\subset U\,$ has a 
	neighborhood with certain favorable attributes. For that, note that 
	$\,|F(\th,s)-\c(h)|\,$ is 
	continuous on the set of triples $\,(\th,s,h)\in\G\times[-a,a]^{2}$, and that 
	$\,|p_{0}-\c(\z)|=2r\,$ for some $\,r>0\,$. So by making $\,\eta>0\,$ small enough, we 
	can ensure two properties:\smallskip
	\begin{align*}
		\text{i)}&\,\ \l|\z\pm \eta\r|<b,\\
		\text{ii)}&\,\ \l|\th-\th_{0}\r|,\,\l|s-\z\r|,\,|h-\z|<\eta\ \Rightarrow\ 
		\l|F(\th,s)-\c(h)\r|>r
	\end{align*}
	
	Now for any $\,(\th,s,h)\,$ in the $\,\eta$-neighborhood of $\,(\th_{0},\z,\z)\,$ 
	defined by (ii) above, consider the unit vector\smallskip
	\[
		w=w(\th,s,h) :=\frac{F(\th,s)-\c(h)}{\l|F(\th,s)-\c(h)\r|}
	\]
	
	Combining (b) from \hyperref[def:gt]{Definition \ref{def:gt}} with (ii), 
	we then have\smallskip
	\begin{equation}\label{eqn:se}
		\l|u\cdot w\r| = \frac{|s-h|}{\l|F(\th,s)-\c(h)\r|} \le \frac{|s-h|}{r}
	\end{equation}
	
	Subtract the $\,w$-component from $\,u\,$ and normalize to construct 
	a unit vector normal to $\,w$:\smallskip
	\begin{equation}\label{eqn:v}
		 v:= \frac{u-(u\cdot w)w}{\l|u-(u\cdot w)w\r|}
	\end{equation}
	
	By design, the plane $\,v_{\c(h)}^{\bot}\,$ now contains $\,F(\th,s)$.
	We shall want it to cut $\,F(U)\,$ along a central loop, and our hypotheses
	certify that, \tit{if} we can show $\,\phi:=\phi(u,v)<\eps$. To do so, 
	combine (ii) with the triangle inequality to deduce $\,|s-h|<2\eta$, and 
	hence\smallskip
	\begin{equation*}
		\sin^{2}\phi 
		= 1-\cos^{2}\phi 
		= 1-(u\cdot v)^{2}
		=(u\cdot w)^{2}
		\le \l|\frac{s-h}{r}\r|^{2}
		\le \frac{4\eta^{2}}{r^{2}}
	\end{equation*} 
	
	Since $\,\phi<\eps\,$ when $\,\sin\phi<\sin\eps$, this yields the bound we
	seek if we require, along with (i) and (ii) above, that\smallskip
	\[
		\text{iii)}\quad 0<\eta<\frac{r\sin\eps}{2}
	\]
	
	Together, restrictions (i), (ii), and (iii) on $\,\eta>0\,$ now leverage our hypotheses
	to ensure that for $\,v\,$ given by \hyperref[eqn:v]{(\ref{eqn:v})}, the plane 
	$\,v_{\c(h)}^{\bot}\,$ contains both $\,F(\th,s)\,$ and $\,\c(h)\,$, and cuts $\,F(U)\,$ 
	along a loop with central symmetry about $\,\c(h)$.
	
	We can now make the main geometric argument for our lemma.

	Consider the mapping that sends $\,(\th,s,h)\,$
	to the reflection of $\,F(\th,s)\in F(U)\,$ through $\,\c(h)$:\smallskip
	\begin{equation}\label{eqn:G}
		(\th,s,h)\ \longmapsto\ 2\c(h)-F(\th,s)
	\end{equation}
	
	Our hypotheses guarantee that for all small enough $\,|\tau|>0$, 
	\tit{the arc parametrized by}
	\[
		\b(\tau):=2\c(h+\tau)-F(\th,s)
	\]
	\tit{stays in $\,F(U)$}. Trivially, its initial velocity is $\,2\,\dot\c(h)$, which
	cannot vanish because $\,u^{*}(\dot\c(h))=1\,$, by condition (b) from 
	\hyperref[def:gt]{Definition \ref{def:gt}}. This proves:
	
	\tit{If $\,\eta>0\,$ satisfies (i), (ii), and (iii) above, then for all $\,(\th,s,h)\,$
	with $\,|\th-\th_{0}|,|s-\z|,|h-\z|<\eta$, the plane tangent to $\,F(U)\,$ at 
	$\,2\,\c(h)-F(\th,s)\,$ contains $\,\dot\c(h)\ne \nv 0$.}
	
	It follows immediately that whenever $\,|t|<\eta$, each tangent plane to $\,F(U)\,$
	in a neighborhood of $\,p_{0}=F(\th_{0},\z)\,$ contains both $\,\dot\c(\z)\,$ and 
	$\,\dot\c(\z+t)$. From this, we can deduce constancy of $\,\dot\c\,$ near $\,\z\,$.
	
	Indeed, we would otherwise have $\,\dot\c(\z+t)\ne\dot\c(\z)\,$ for some $\,t\in(-\eta,\eta)$,
	and since they have the same $\,u$-component, by (b) from 
	\hyperref[def:gt]{Definition \ref{def:gt}}, inequality means independence.
	Since $\,p_{0}\,$ has a neighborhood in $\,F(U)\,$ where {every} tangent plane contains---hence 
	is spanned by---these same two non-zero vectors, independence forces constancy of
	the unit normal to $\,F(U)\,$ near $\,p_{0}\,$. A neighborhood of $\,p_{0}\,$ in $\,F(U)\,$ 
	then lies in a plane---a plane cutting $\,u_{p_{0}}^{\bot}\,$ along a line. The 
	cross-cut parametrized by $\,F(\,\cdot\,,\z)\,$ must contain a segment of that line, 
	with $\,p_{0}=F(\th_{0},\z)\,$ in its interior. But we maximized $\,|F(\th,\z)|^{2}\,$ 
	at $\,\th_{0}$, and $\,x\mapsto|x|^{2}\,$ is strictly convex; it \tit{cannot} reach 
	a local max on the interior of a segment. We have thus contradicted the possibility that 
	$\,\dot\c(\z+t)\ne\dot\c(\z)\,$ for any $\,|t|<\eta$. It follows that
	$\,\dot\c\equiv \dot\c(\z)\,$ on a neighborhood of $\,\z$.
		
	Because $\,\z\in(-b,b)\,$ was arbitrary, however, this (and continuity of $\,\dot\c$) 
	yields local constancy of $\,\dot\c\,$ on subset of $\,[-b,b]\,$ that is simultaneously 
	non-empty, open, and closed. The conclusion of our Lemma follows at once.
\end{proof}

\begin{cor}[Local cylinder]\label{cor:locCyl}
	Under the assumptions of \hyperref[lem:axis]{Lemma \ref{lem:axis}}, 
	$\,F(U)\,$ is a central cylinder.
\end{cor}

\begin{proof} %
	By \hyperref[lem:axis]{Lemma \ref{lem:axis}}, the central curve of $\,F(U)$---what we
	shall henceforth call its \tit{axis}---is a line segment parallel to $\,v:=\dot\c(0)\,$. 
	The Corollary follows easily from one additional
	
	\tit{Claim.} \tit{Every tangent plane to $\,F(U)\,$ contains $\,v\,$.}
	
	As $\,F(U)\,$ is closed, it suffices to prove this for an arbitrary point $\,p\in F(U)\,$ 
	\tit{not} lying on its axis. Let $\,h:=u_{c}^{*}(p)\,$ denote the signed height of 
	$\,p\,$ above $\,u_{c}^{\bot}\,$ ($c=\c(0)$).
	
	The cross-cut of $\,F(U)\,$ parallel to $\,u_{c}^{\bot}\,$ at height $\,h\,$ is
	central about $\,\c(h)$, so both $\,p\,$ and its reflection $\,q:=2\c(h)-p\,$ 
	lie in $\,F(U)\cap u_{\c(h)}^{\bot}$. Like $\,p$, of course, 
	$\,q\,$ avoids the axis of $\,F(U)$. 
	
	Our assumptions say that slightly tilted cross-cuts of $\,F(U)\,$ are \tit{also} 
	central about the axis, and provided $\,|t|>0\,$ is sufficiently small, some such cross-cut
	contains both $\,q\,$ and $\,\c(h+t)\,$.
	
	It follows that the arc $\,t\mapsto 2\mu(h+t)-q\,$ lies in $\,F(U)\,$ for all 
	sufficiently small $\,|t|$. The resulting differentiable arc passes through $\,p\,$ 
	when $\,t=0$, with initial velocity $\,2\dot\c(h)=2v$, so $\,v\,$ is tangent to $\,F(U)\,$ 
	at $\,p$, as our Claim proposes.
	
	The corollary quickly follows:  $F(U)\,$ is everywhere tangent to the constant 
	vectorfield $\,v$, so it is foliated by line segments parallel to $\,v$. This makes it a 
	(generalized) cylinder, and it has a compact central cross-cut, so it is {central} too.
\end{proof}

By chaining together intermediate results from above, we can quickly prove our main theorem.
We restate it here for the reader's convenience. 

\hyperref[thm:8case]{\textbf{Theorem \ref{thm:8case}}}:
	\tit{If $\,F\colon M\to\R^{3}\,$ is a complete $C^{2}$ immersion with \cx, 
	and some plane in general position relative to $\,F\,$ cuts it along a clean figure-8, 
	then $F(M)$ is a central cylinder.}

\begin{proof}[Proof of Main Result]
	We are assuming that for some plane $\,P=u_{c}^{\bot}\,$ in general position 
	relative to $\,F$, a clean cross-cut $\,\G\subset F^{-1}(P)\,$ whose image 
	$\,\g:=F(\G)\subset P\,$ is a clean figure-8. As discussed in connection with 
	\hyperref[def:gt]{Definition \ref{def:gt}}, that puts $\,\G\,$ in the image of a good 
	tubular patch $\,(U,\psi)\,$.
	
	\hyperref[lem:gt]{Lemma \ref{lem:gt}} now provides some $\,0<\eps<a\,$ 
	for which every cross-cut $\,F^{-1}(v_{\c(h)}^{\bot})\cap U\,$ is a clean figure-8 
	when $\,|h|<\eps\,$ and $\,\phi(v,u)<\eps$. As above, $\,\c\:[-a,a]\to\R^{3}\,$ here
	denotes the central curve of $\,F(U)\,$.
	
	\hyperref[lem:cdnm]{Lemma \ref{lem:cdnm}} now certifies that each of these cross-cuts 
	has central symmetry about $\,\c(h)$, and \hyperref[cor:locCyl]{Corollary \ref{cor:locCyl}} 
	(with $\,b=\eps$) then shows that, within the	$\eps$-neighborhood $\,P_{\eps}$ of $\,P$, 
	the image $\,F(U)\,$ is a central cylinder.
	
	A simple open/closed argument now shows that $\,F(M)\,$ is the complete extension of 
	that cylinder.
	
	Indeed, call a scalar $\,a>0\,$ \tit{reachable} if there exists a good tubular patch 
	$\,\psi\:\G\times[-a,a]\to M\,$ whose image is mapped by $\,F\,$ to a central cylinder. 
	Given what we have just proven, we know that\smallskip
	\[
		A:=\sup\l\{a>0\: \text{$a$ is reachable}\r\} \ge\eps>0
	\]
	Our theorem amounts to the assertion $\,A=\infty\,$, which we can now establish by
	contradiction. For if $\,A<\infty$, the completeness and smoothness of $\,F\,$ 
	would let us construct a \tit{maximal} good tubular patch $\,\psi\:\G\times[-A,A]\to 
	M$, with $\,F\circ\psi\,$ mapping $\,\G\times[-A,A]\,$ to a central cylinder with 
	boundary in $\,\bd P_{A}\,$. The two loops bounding this cylinder would clearly be 
	clean figure-8's. By applying the argument above, however, we could deduce that their 
	preimages in $\,M\,$ each have good tubular neighborhoods mapping to central cylinders
	via $\,F\,$. Our supposedly {maximal} good tubular patch could then be extended 
	slightly at each boundary component, violating the maximality of $\,A\,$. 
	Thus, $\,A\,$ \tit{cannot} be finite; we have $\,A=\infty$ which gives our theorem.
\end{proof}

%%%%%%%%%%%%%%%  ACKNOWLEDGMENTS  %%%%%%%%%%%%%%%%%

%\section*{Acknowledgments} 

%%%%%

%\vfill
\end{document}